\documentclass{aims}
\usepackage{amsmath}
\usepackage{enumitem}
\usepackage{graphics} 
\usepackage{epsfig} 
\usepackage{graphicx}  \usepackage{epstopdf}
\usepackage[colorlinks=true]{hyperref}
\hypersetup{urlcolor=blue, citecolor=red}

\def\currentvolume{x}
\def\currentissue{x}
\def\currentyear{20xx}
\def\currentmonth{xx}
\def\ppages{x--xx}
\def\DOI{10.3934/xx.xx.xxxx}

\newtheorem{theorem}{Theorem}[section]
\newtheorem{corollary}{Corollary}

\newtheorem{lemma}[theorem]{Lemma}
\newtheorem{proposition}{Proposition}

\theoremstyle{definition}
\newtheorem{definition}[theorem]{Definition}
\newtheorem{remark}{Remark}
\newtheorem{assumption}[theorem]{Assumption}

\newcommand{\N}{\mathbb N}
\newcommand{\Z}{\mathbb Z}

\newcommand{\R}{\mathbb R}

\newcommand{\T}{\mathbb T}

\newcommand{\Cs}{\mathcal C}
\newcommand{\Ps}{\mathcal P}
\newcommand{\diam}{\mathrm{diam}}
\newcommand{\Diff}{\mathrm{Diff}}

\newcommand{\union}{\mathop{\bigcup}}
\newcommand{\cl}{\mathrm{cl}}
\newcommand{\htop}{h_{\mathrm{top}}}

\title[Equilibrium States of Almost Anosov Diffeomorphisms] 
{Equilibrium States of Almost Anosov Diffeomorphisms}

\author[Dominic Veconi]{}


\email{dkv5049@psu.edu}

\begin{document}
	\maketitle
	
	\centerline{\scshape Dominic Veconi}
	\medskip
	{\footnotesize
		\centerline{Department of Mathematics}
		\centerline{Pennsylvania State University}
		\centerline{ University Park, PA 16802, USA}
	} 
	
	
	
	\bigskip
	

	\begin{abstract}
		We develop a thermodynamic formalism for a class of diffeomorphisms of a torus that are ``almost-Anosov''. In particular, we use a Young tower construction to prove the existence and uniqueness of equilibrium states for a collection of non-H\"older continuous geometric potentials over almost Anosov systems with an indifferent fixed point, as well as prove exponential decay of correlations and the central limit theorem for these equilibrium measures. 
	\end{abstract}
	
	\section{Introduction}
	
	In \cite{Hu00}, Hu gave conditions on the existence of Sinai-Ruelle-Bowen (SRB) measures for a class of surface diffeomorphisms that are hyperbolic everywhere except at a finite set of indifferent fixed points (that is, fixed points $p$ of the map $f : M \to M$ such that $Df_p = \mathrm{id}$). Maps that are uniformly hyperbolic away from finitely many points are known as ``almost Anosov'' diffeomorphisms, and they provide a class of nonuniformly hyperbolic diffeomorphisms that are, in a sense, as close to being uniformly hyperbolic as one can ask for. With this in mind, when one would like to investigate certain properties of nonuniformly hyperbolic systems, a good starting point may be almost Anosov and almost hyperbolic maps due to the structure built into the definition of these maps. (See \cite{AL13}, for example, on ``almost axiom A diffeomorphisms''.) 
	
	Among the similarities between Anosov and almost Anosov diffeomorphisms is the fact that the tangent bundle has a splitting $TM = E^u \oplus E^s$ into stable and unstable subspaces, and except at the singularities of the almost Anosov diffeomorphism, this splitting is continuous. This was proven in \cite{Hu00} as part of the proof that almost Anosov diffeomorphisms admit SRB measures. In particular, the \emph{geometric $t$-potentials} $\phi_t(x) = -t\log\left|Df_x|_{E^u(x)}\right|$ are well-defined for almost Anosov diffeomorphisms, but may not be H\"older continuous at the indifferent fixed points. In this paper, we discuss equilibrium measures for almost Anosov diffeomorphisms, using $\phi_t$ as our potential function. 
	
	
	One of the more well-known explicit examples of a nonuniformly hyperbolic surface diffeomorphism with an indifferent fixed point is the Katok map of the torus, introduced in \cite{Kat79}. The authors of \cite{PSZ17} effected thermodynamic formalism for this map by proving the existence and uniqueness of equilibrium measures for geometric $t$-potentials. Their techniques were largely based on arguments developed in \cite{PSZ16}, which discussed thermodynamics of diffeomorhpisms admitting Young towers. It should be noted that the Katok map is not an example of an almost Anosov diffeomorphism, because the definition of almost Anosov diffeomorphisms includes strong conditions on the stable and unstable cones that are not satisfied by the Katok map (specifically condition (ii) of definition \ref{AAD-def}). However, the Young tower construction in \cite{PSZ17} for the Katok map provides a framework to effect thermodynamics for a broad class of nonuniformly hyperbolic maps with finitely many non-hyperbolic fixed points. In particular, our arguments use  the results of \cite{PSZ16} as well by showing that almost Anosov surface diffeomorphisms with an indifferent fixed point admit Young towers, and concluding that certain geometric $t$-potentials uniquely admit equilibrium measures. 
	
	The almost Anosov maps we consider are somewhat restricted. For example, an almost Anosov diffeomorphism may have non-hyperbolic fixed points that are not indifferent. If $p \in M$ is a non-hyperbolic non-indifferent fixed point, there are three possibilities for the differential $Df_p$: 
	\begin{enumerate}[label=(\arabic*)] 
		\item the differential has one eigenvalue of 1 and another eigenvalue $0 < \lambda < 1$; 
		\item the differential has one eigenvalue of 1 and another eigenvalue $\lambda > 1$;
		\item the differential is non-diagonalizable. 
	\end{enumerate}
	In \cite{HY95}, an example of case (1) was discussed and shown to not admit an SRB measure. Rather, this map admits what is referred to as an ``infinite SRB measure'', or a measure $\mu$ whose conditional measures on unstable leaves are absolutely continuous with respect to Lebesgue, $\mu(M \setminus U) < \infty$ for any neighborhood $U$ of the set of non-hyperbolic fixed points, and $\mu(M) = \infty$. In \cite{Hu00}, the author makes the heuristic observation that even for almost Anosov maps with indifferent fixed points, if the differentials exhibit stronger contraction than expansion as points approach indifferent singularities, one may expect to find this map admits an infinite SRB measure rather than an SRB probability measure. On the other hand, when an almost Anosov map with an indifferent fixed point has differentials exhibiting stronger expansion than contraction as points approach the singularity, one may expect this map to admit an SRB probability measure. Accordingly, it would be straightforward to show that almost Anosov diffeomorphisms exhibiting behavior as in (2) of the above list will admit an SRB measure, and we believe further demonstrating the existence of equilibrium states would also be fairly direct. 
	
	Additionally, the only almost Anosov systems we consider are almost Anosov diffeomorphisms of the two-dimensional torus $\T^2$. This is because our arguments rely on a topological conjugacy between an almost Anosov map $f$ and a uniformly hyperbolic Anosov diffeomorphism $\tilde f$ of the surface, and the only surface that admits Anosov diffeomorphisms is $\T^2$. It is widely believed that almost Anosov diffeomorphisms on manifolds of any dimension are topologically conjugate to Anosov diffeomorphisms, but as far as the author of this paper has observed, an explicit proof of this is absent from the literature. We propose the following amendment to this conjecture: that an almost Anosov diffeomorphisms on a manifold admitting Anosov diffeomorphisms is topologically conjugate to an Anosov diffeomorphism. This paper includes a proof of this conjecture for the two-dimensional torus, which is the first published proof of such a result that we know of.
	
	This paper is structured as follows. In Section 2, we define almost Anosov diffeomorphisms and describe some important dynamic and topological properties of these maps, including the decomposition of the tangent bundle into stable and unstable subspaces, as well as prove the conjugacy between almost Anosov maps $f : \T^2 \to \T^2$ and Anosov maps $\tilde f$ of the torus. In Section 3, we discuss certain thermodynamic constructions (including equilibrium states and topological pressure), as well as state our main results. Section 4 describes thermodynamics of Young's diffeomorphisms, establishing techniques used in \cite{PSZ16} and \cite{PSZ17}. In Section 5 we explicitly construct a Young tower for our toral almost Anosov diffeomorphism. Finally in Section 6 we use the techniques in Section 4 to prove our main result. 
	
	
	\section{Preliminaries}
	
	\begin{definition}\label{AAD-def}
		A diffeomorphism $f : M \to M$ of a Riemannian manifold $M$ is \emph{almost Anosov} if there is a continuous family of cones $x \mapsto \Cs^s_x$ and $x \mapsto \Cs^u_x$ (with $x \in M$) such that, except for a finite set $S \subset M$, 
		\begin{enumerate}[label=(\roman*)]
			\item $Df_x \Cs_x^s \supseteq \Cs_{fx}^s$ and $Df_x \Cs_x^u \subseteq \Cs_{fx}^u$; 
			\item $\left| Df_x v\right| < |v|$ for $v \in \Cs_x^s$, and $\left|Df_x v\right| > |v|$ for $v \in \Cs_x^u$.  
		\end{enumerate}
	\end{definition}
	
	\begin{remark}
		By continuity of the cone families, for any $p \in S$, we have: 
		\begin{enumerate}[label=(\roman*)]
			\item $Df_p \Cs_p^s \supseteq \Cs_{fp}^s$ and $Df_p \Cs_p^u \subseteq \Cs_{fp}^u$; 
			\item $\left| Df_p v\right| \leq |v|$ for $v \in \Cs_p^s$, and $\left|Df_p v\right| \geq|v|$ for $v \in \Cs_p^u$.  
		\end{enumerate}
		For simplicity, we assume $S$ is invariant (note if $fp \not\in S$ for some $p \in S$, then in fact condition (ii) in the definition holds for $p$, so we may consider $S \setminus \{p\}$). We further assume $fp = p$ for all $p \in S$; if not, we replace $f$ by $f^n$ for appropriate $n\in \N$. 
	\end{remark}
	
	Given a subset $A \subseteq M$ and a distance $r > 0$, define: 
	\[
	B_{r}(A) = \left\{ x \in M : d(x,A) < r\right\} 
	\]
	where $d(x,A)$ is the Riemannian distance from $x$ to the set $A$. In general, we assume $f \in \Diff^4(M)$ for reasons of technical approximations. Additionally, because we may have $\left| Df_x v\right|/|v| \to 1$ for $v \in \Cs_x^s$ or for $v \in \Cs_x^u$ as $x$ approaches $S$, we make the following definition to control the speed at which $\left| Df_x v\right|/|v| \to 1$: 
	
	\begin{definition}
		An almost Anosov diffeomorphism $f$ is \emph{nondegenerate} (up to third order) if there are constants $r_0, \kappa^s, \kappa^u > 0$ such that for all $x \not\in B_{r_0}(S)$, 
		\begin{align*}
		\left| Df_x v \right| &\geq \left( 1 + \kappa^u d(x, S)^2\right)|v| \qquad \forall v \in \Cs_x^u, \\
		\left| Df_x v \right| &\leq \left( 1 - \kappa^s d(x,S)^2\right)|v| \qquad \forall v \in \Cs_x^s.
		\end{align*}
	\end{definition}
	
	\begin{remark}
		In general, for an almost Anosov diffeomorphism $f : M \to M$, for any $r > 0$, there are constants $K^s = K^s(r)$ and $K^u = K^u(r)$ so that, for $x \not\in B_r(S)$, 
		\begin{align*}
		|Df_x v| &\geq K^u|v| \quad \forall v \in \Cs^u_x, \qquad \textrm{and} \qquad |Df_x v| \leq K^s|v| \quad \forall v \in \Cs^s_x.
		\end{align*}
	\end{remark}
	
	Define the \emph{local stable and unstable manifolds} at the point $x \in M$: 
	\begin{align*}
	W_\epsilon^u(x) &= \left\{ y \in M : d\left(f^{-n}y, f^{-n}x\right) \leq \epsilon \quad \forall n \geq 0 \right\}, \\
	W_\epsilon^s(x) &= \left\{ y \in M : d\left(f^{n}y, f^{n}x\right) \leq \epsilon \quad \forall n \geq 0 \right\}.
	\end{align*}
	
	\begin{theorem}\cite{Hu00}\label{bundle-split}
		There exists an invariant decomposition of the tangent bundle into $TM = E^u \oplus E^s$ such that for every $x \in M$: 
		\begin{itemize}
			\item $E^\eta_x \subseteq \Cs^\eta_x$ for $\eta = s, u$;
			\item $Df_x E^\eta_x = E^\eta(fx)$ for $\eta = s, u$; 
			\item $W_\epsilon^\eta(x)$ is a $C^1$ curve, which is tangent to $E^\eta(x)$ for $\eta = s, u$. 
		\end{itemize}
		Furthermore, the decomposition $TM = E^u \oplus E^s$ is continuous everywhere except possibly on $S$. 
	\end{theorem}
	
	\begin{remark}
		The proof of this theorem in \cite{Hu00} gives tangency of $W^u_\epsilon(x)$ to $E^u(x)$. The author notes that $W^s_\epsilon(x)$ is tangent to $E^s_x$, and the same argument can be used to prove this as was used to prove the fact for $W^u_\epsilon(x)$ and $E^u_x$ by considering $f^{-1}$ instead of $f$. 
	\end{remark}
	
	\begin{remark}\label{localproduct}
		Since $E^\eta_x \subseteq \Cs^\eta_x$ for $x \in M$ and $\eta = s, u$, despite the fact that the decomposition $TM = E^u \oplus E^s$ is possibly discontinuous at points in $S$, $f$ has local product structure. Specifically, there are constants $\epsilon_0$, $\delta_0 > 0$ such that for every $x, y \in M$ with $d(x,y) < \delta_0$, the intersection $W^u_{\epsilon_0}(x) \cap W^s_{\epsilon_0}(y)$ contains exactly one point, which we denote $[x,y]$. 
	\end{remark}
	
	From here, we assume $M = \T^2$, $f : \T^2 \to \T^2$ is almost Anosov with singular set $S = \{0\}$, and that $Df_0 = \mathrm{Id}$. It is shown in \cite{Hu00} that nondegeneracy of $f$ implies $D^2 f_0 = 0$, so there is a coordinate system around 0 for which $f$ is expressible as 
	\begin{equation}\label{0-coords}
	f(x,y) = \bigg(\begin{array}{cc} x\big(1+\phi(x,y)\big), & y\big(1-\psi(x,y)\big)\end{array}\bigg),
	\end{equation}
	for $(x,y) \in \R^2$ and 
	\begin{align*}
	\phi(x,y) &= a_0 x^2 + a_1 xy + a_2 y^2 + O\left( |(x,y)|^3\right), \\
	\psi(x,y) &= b_0 x^2 + b_1 xy + b_2 y^2 + O\left( |(x,y)|^3\right),
	\end{align*}
	where $|(x,y)| := \sqrt{x^2 + y^2}$ for $x,y \in \R$. It is further shown in \cite{Hu00} that when there is $\alpha \in (0,1)$ so that $\alpha a_2 > 2b_2$, $a_1 = b_1 = 0$, and $a_0 b_2 - a_2 b_0 > 0$, the map $f$ admits an SRB measure. Henceforth, we shall assume \begin{equation}\label{simple-phi-psi}
	\phi(x,y) = ax^2 + by^2 \quad \textrm{and} \quad \psi(x,y) = cx^2 + dy^2 
	\end{equation}
	for some $a,b,c,d \in (0,\infty)$. 
	
	We begin by showing that almost Anosov diffeomorphisms on the torus are topologically conjugate to Anosov diffeomorphisms. In particular, this allows us to construct Markov partitions for almost Anosov diffeomorphisms of arbitrarily small diameter. Our proof requires the following result: 
	
	\begin{theorem}\cite{Gu75}
		Suppose an expansive map $f : M \to M$ of a Riemannian manifold is a $C^0$ limit of Anosov diffeomorphisms. Then $f$ is topologically conjugate to an Anosov diffeomorphism. 
	\end{theorem}
	
	We make the following technical assumption to help us prove that Anosov maps are topologically conjugate to almost Anosov maps, as well as to show exponential decay of correlations and the central limit theorem. (Although this assumption is necessary for decay of correlations and CLT, we believe that one could prove topological conjugacy under a weaker assumption). 
	
	\begin{assumption}\label{A} There are constants $r_0$ and $r_1$, with  $0 < r_0 < r_1$ for which the almost Anosov map $f : \T^2 \to \T^2$ is equal to a linear Anosov map $\tilde f : \T^2 \to \T^2$ outside of $B_{r_1}(0)$, and within $B_{r_0}(0)$, $f$ has the form (\ref{0-coords}). 
	\end{assumption}
	
	Maps satisfying this assumption do exist. To construct such a diffeomorphism, one can choose a hyperbolic matrix $A \in \mathrm{SL}(2,\Z)$ and a smooth bump function $\omega : \R^2 \to \R$ supported on $B_{r_1}(0)$ and equal to 1 on $B_{r_0}(0)$. Define the map $\tilde \Phi : \R^2 \to \R^2$ by 
	\[
	\tilde\Phi(x,y) = \bigg( \begin{array}{cc} x\big(1+\phi(x,y)\big), & y \big(1-\psi(x,y)\big) \end{array}\bigg) - A(x,y),
	\]
	where $\phi$ and $\psi$ are as in (\ref{simple-phi-psi}), and let $\Phi : \T^2 \to \T^2$ be the quotient of the map $\omega\tilde \Phi$ by $\Z^2$. The map $f = \Phi + A$ is an example of an almost Anosov diffeomorphism satisfying Assumption \ref{A}. 
	
	\begin{theorem}\label{Anosov-conjugacy}
		A nondegenerate almost Anosov diffeomorphism $f : \T^2 \to \T^2$ satisfying the above assumption is topologically conjugate to an Anosov diffeomorphism. 
	\end{theorem}
	
	\begin{proof}
		We need to show that $f$ is both expansive and is the limit of a sequence of Anosov diffeomorphisms. The argument for expansiveness is standard: suppose for every $\epsilon > 0$ we have distinct $p,q \in \T^2$ so that $d\left(f^n p, f^n q \right) < \epsilon$ for every $n \in \Z$. Then condition (i) of definition \ref{AAD-def} implies $q \in W^u_\epsilon(p) \cap W^s_\epsilon(p)$ (see remark \ref{localproduct}). If we assume $\epsilon < \min\{\delta_0, \epsilon_0\}$ (where $\epsilon_0, \delta_0$ are as in remark \ref{localproduct}), this gives us $W^\eta_{\epsilon_0}(p) \supseteq W^\eta_\epsilon(p)$ for $\eta = s, u$. In particular, 
		\[
		q \in W^u_\epsilon(p) \cap W^s_\epsilon(p) \subseteq W^u_{\epsilon_0}(p) \cap W^s_\epsilon(p) = \{p\}.
		\]
		Therefore $p=q$, contradiction. 
		
		To prove that $f$ is a limit of a sequence of Anosov diffeomorphisms, we define a homotopy $H : \T^2 \times [0,1] \to \T^2$ so that $H_0 = f$, and $H_\epsilon$ is Anosov for every $\epsilon \in (0,1]$. For $r_0 > 0$ small, the disc $B_{r_0}$ is a coordinate chart in which $f$ has the form 
		\[
		f(x,y) = \bigg(\begin{array}{cc} x\big(1+\phi(x,y)\big), & y\big(1-\psi(x,y)\big)\end{array}\bigg),
		\]
		for $(x,y) \in \R^2$ and 
		\begin{align*}
		\phi(x,y) &= ax^2 + by^2+ O\left( |(x,y)|^3\right), \\
		\psi(x,y) &= cx^2 + dy^2 + O\left( |(x,y)|^3\right).
		\end{align*}
		The differentials in $B_{r_0}$ are of the form
		\[
		Df(x,y) = \left( \begin{array}{cc}
		1 + 3ax^2 + by^2 & 2bxy \\
		-2cxy & 1 - cx^2 - 3dy^2
		\end{array} \right) + \mathcal{O}
		\]
		where $\mathcal{O}$ is a matrix of terms of order $|(x,y)|^3$. By condition (ii) in definition \ref{AAD-def}, these matrices are hyperbolic; that is, they have two real eigenvalues, one whose magnitude is greater than 1 and one whose magnitude is less than 1. After taking the radius $r_0$ to be sufficiently small, the terms of $\mathcal{O}$ are very close to 0, so we may assume $\mathcal{O} = 0$. Moreover, since hyperbolicity is an open condition on matrices, for each $(x,y) \in B_{r_0}\setminus\{0\}$, there are $\pi(x,y) > 0$ and $\rho(x,y) > 0$ for which the matrix 
		\begin{equation}\label{badDH}
		\left( \begin{array}{cc}
		1 + \left( 3a - \alpha\right) x^2 + by^2 & 2bxy \\
		-2cxy & 1 - cx^2 - (3d - \beta)y^2
		\end{array} \right)
		\end{equation}
		is hyperbolic for $0 \leq \alpha \leq \pi(x,y)$ and $0 \leq \beta \leq \rho(x,y)$. We may assume these functions $\pi, \rho : B_{r_0} \to \R$ are continuous and that $\pi(0,0) = \rho(0,0) = 0$. Define the nonnegative functions $\alpha, \beta : [0,1] \to \R$ by
		\[
		\alpha(s) = \inf_{x^2 + y^2 = s^2} \pi(x,y) \quad \textrm{and} \quad \beta(s) = \inf_{x^2 + y^2 = s^2} \rho(x,y).
		\] 
		Further define the family of nonnegative continuous maps $g_\epsilon, h_\epsilon : [0,1] \to \R$ for each $\epsilon \in [0,r_0]$ so that $g_\epsilon(t) = h_\epsilon(t) = 0$ for $t \geq \epsilon^2$, and for $t < \epsilon^2$ we let:
		\[
		g_\epsilon(t) = \frac 1 4 \int_t^{\epsilon^2} \alpha\left(\sqrt u\right) \, du \quad \textrm{and} \quad h_\epsilon(t) = \frac 1 4 \int_t^{\epsilon^2} \beta\left(\sqrt u \right) \, du.
		\]
		Observe that for $s < \epsilon$, we have:
		\begin{equation}\label{gh-approx}
		-\frac 1 2 \alpha(s) < g'_{\epsilon}\left(s^2\right) < 0 \quad \textrm{and} \quad -\frac 1 2 \beta(s) < h'_{\epsilon}\left(s^2\right)< 0,
		\end{equation}
		which further implies that $g_\epsilon \to 0$ and $h_\epsilon \to 0$ in $C^0$ as $\epsilon \to 0$. Let $H_\epsilon : \T^2 \to \T^2$ be maps for each $\epsilon > 0$ so that in the coordinate ball $B_{r_0}$, $H_\epsilon$ is of the form
		\begin{align*}
		H_\epsilon(x,y) &= \bigg( x\left(1+g_\epsilon\left(x^2 + y^2\right) + ax^2 + by^2 \right), \\
		&\qquad\qquad y \left( 1 - h_\epsilon \left(x^2 + y^2 \right) - cx^2 - dy^2 \right) \bigg),
		\end{align*}
		and we further assume that outside of $B_{r_1}$, the map $H_\epsilon =  F$ is linear Anosov for all $\epsilon$, and in the annulus $B_{r_1} \setminus B_{r_0}$, $H_\epsilon$ smooths out to $F$. We see that $f$ is the $C^0$ limit of $H_\epsilon$ as $\epsilon \to 0$, so we only need to show each $H_\epsilon$ is Anosov for $\epsilon \in (0, \epsilon_0)$ for some small $\epsilon_0$. To that end, the derivative of $H_\epsilon : \T^2 \to \T^2$ for a fixed $\epsilon$ is 
		\begin{equation}\label{DH}
		DH_\epsilon(x,y) = \left( \begin{array}{cc}
		1 + \Phi_\epsilon(x,y) & 2xy\left(g_\epsilon'\left(x^2 + y^2 \right) + b \right) \\ 
		-2xy\left(h_\epsilon' \left(x^2 + y^2 \right) + c \right) & 1 - \Psi_\epsilon(x,y)
		\end{array}\right) 
		\end{equation}
		in the neighborhood $B_{r_0}$ of the origin, where
		\begin{align*}
		\Phi_\epsilon(x,y) &= g_\epsilon\left(x^2 + y^2 \right) + \big(3a + 2 g'_\epsilon\left(x^2 + y^2 \right) \big) x^2 + by^2, \\
		\Psi_\epsilon(x,y) &= h_\epsilon\left(x^2 + y^2 \right) + \big( 3d + 2 h_\epsilon'\left(x^2 + y^2 \right)\big) y^2 + cx^2 .
		\end{align*}
		Our objective is to show that these linear maps are hyperbolic for every $(x,y) \in \T^2$. Since $H_\epsilon(x,y) = F$ for $(x,y) \not\in B_\epsilon$, and we know $DF$ is hyperbolic everywhere, we only need to check the case when $x^2 + y^2 < \epsilon^2$. 
		
		Hyperbolicity may fail in two ways: if $\Phi_\epsilon(x,y)$ or $\Psi_\epsilon(x,y)$ are too small at some point $(x,y)$, or if the upper right and lower left terms of (\ref{DH}) are too far from 0. We first address the latter concern. Assuming $|(x,y)|$ is small, since $\alpha\left(x^2 + y^2 \right) \to 0$ and $\beta \left(x^2 + y^2 \right)\to 0$ as $|(x,y)| \to 0$, equation (\ref{gh-approx}) gives us $ g'_\epsilon\left(x^2 + y^2 \right) > -b$ and $h'_\epsilon(\left(x^2 + y^2\right) > -c$. In particular, the furthest either the upper right or lower left entries of (\ref{DH}) can be from 0 are $2bxy$ and $-2cxy$ for any $x,y$. 
		
		To address the first concern, i.e. ensuring $\Phi_\epsilon(x,y)$ and $\Psi_\epsilon(x,y)$ are not too small, we observe:
		\begin{align*}
		\Phi_\epsilon(x,y) &\geq g_\epsilon \left(x^2 + y^2 \right) + \big(3a - \alpha\left(|(x,y)|\right)\big) x^2 + by^2 > \big( 3a - \pi(x,y)\big) x^2 + by^2,
		\end{align*}
		and 
		\begin{align*}
		\Psi_\epsilon(x,y) &\geq h_\epsilon \left(x^2 + y^2 \right) + \big(3d - \beta(|(x,y)|) \big) y^2 + cx^2 > \big( 3d - \rho(x,y) \big) y^2 + cx^2.
		\end{align*}
		Therefore, the furthest each linear map $DH_\epsilon(x,y)$ could possibly be from being hyperbolic would be if $DH_\epsilon(x,y)$ were of the form 
		\[
		\left( \begin{array}{cc}
		1 + \big( 3a - \pi(x,y)\big) x^2 + by^2 & 2bxy \\ -2cxy & 1 - \big( 3d - \rho(x,y) \big) y^2 - cx^2
		\end{array}\right) 
		\]
		and as we saw in (\ref{badDH}), this matrix is still hyperbolic. 
	\end{proof}
	
	\begin{corollary}
		Almost Anosov diffeomorphisms admit Markov partitions of arbitrarily small diameter. 
	\end{corollary}
	
	
	\section{Statement of Results}
	
	Recall from Theorem \ref{bundle-split} \ that there is an invariant decomposition of the tangent bundle $TM = E^u \oplus E^s$ into unstable and stable subspaces, so that $E^u(x)$ and $E^s(x)$ respectively are tangent to the local unstable and stable manifolds $W^u_\epsilon(x)$ and $W^s_\epsilon(x)$. Define the \emph{geometric $t$-potential} $\phi_t(x) = -t\log\big| DF|_{E^u(x)}\big|$. Given a continuous potential function $\phi : M \to \R$, a probability measure $\mu_\phi$ on $M$ is an \emph{equilibrium measure} for $\phi$ if 
	\[
	P_f(\phi) = h_{\mu_\phi}(f) + \int_M \phi \, d\mu_\phi,
	\]
	where $h_{\mu_\phi}(f)$ is the metric entropy of $(M, f)$ with respect to $\mu_\phi$, and $P_f(\phi)$ is the topological pressure of $\phi$; that is, $P_f(\phi)$ is the supremum of $h_\mu(f) + \int_M \phi \, d\mu$ over all $f$-invariant probability measures $\mu$. We denote $\mu_t := \mu_{\phi_t}$. 
	
	Additionally, we say that $f$ has \emph{exponential decay of correlations} with respect to a measure $\mu \in \mathcal{M}(f, M)$ and a class of functions $\mathcal H$ on $M$ if there exists $\kappa \in (0,1)$ such that for any $h_1, h_2 \in \mathcal H$, 
	\[
	\left| \int h_1\left(f^n(x)\right) h_2(x) \, d\mu(x) - \int h_1(x) \, d\mu(x) \int h_2(x)\,d\mu(x)\right| \leq C\kappa^n
	\]
	for some $C = C(h_1, h_2) > 0$. Furthermore, $f$ is said to satisfy the \emph{Central Limit Theorem} (CLT) for a class $\mathcal H$ of functions if for any $h \in \mathcal H$ that is not a coboundary (ie. $h \neq h' \circ f - h'$ for any $h' \in \mathcal H$), there exists $\sigma > 0$ such that 
	\[
	\lim_{n \to \infty} \mu\bigg\{\frac{1}{\sqrt n} \sum_{i=0}^{n-1} \Big(h(f^i(x)) - \int h \, d\mu \Big) < t \bigg\} = \frac 1{\sigma \sqrt{2\pi}} \int_{-\infty}^{t} e^{-\tau^2/2\sigma^2} \, d\tau.
	\]
	
	We now state our main result. 
	
	\begin{theorem}\label{main-theorem}
		Given a transitive nondegenerate almost Anosov diffeomorphism $f : \T^2 \to \T^2$ satisfying Assumption \ref{A} for which $r_1$ is sufficiently small, the following statements hold: 
		\begin{enumerate}
			\item There is a $t_0 < 0$ so that for any $t \in (t_0, 1)$, there is a unique equilibrium measure $\mu_t$ associated to $\phi_t$. This equilibrium measure has exponential decay of correlations and satisfies the central limit theorem with respect to a class of functions containing all H\"older continuous functions on $\T^2$. 
			\item For $t=1$, there are two equilibrium measures associated to $\phi_1$: the Dirac measure $\delta_0$ centered at the origin, and a unique invariant SRB measure $\mu$. If $f$ is Lebesgue-area preserving, this SRB measure coincides with Lebesgue measure. 
			\item For $t > 1$, $\delta_0$ is the unique equilibrium measure associated to $\phi_t$. 
		\end{enumerate}
	\end{theorem}
	
	\begin{remark}
		Uniqueness of $\mu_t$ for $t \in (t_0, 1)$ implies that this equilibrium measure is ergodic. Since the correlations decay, in fact $\mu_t$ is mixing. 
	\end{remark}
	
	\section{Thermodynamics of Young's diffeomorphisms}
	
	To prove the existence of equilibrium measures associated to $\phi_t$ for nondegenerate almost Anosov maps $f$ on $\T^2$, we begin by showing that they are \emph{Young's diffeomorphisms}. We now define this class of maps. 
	
	Given a $C^{1+\alpha}$ diffeomorphism $f$ on a compact Riemannian manifold $M$, we call an embedded $C^1$ disc $\gamma \subset M$ an \emph{unstable disc} (resp. \emph{stable disc}) if for all $x, y \in \gamma$, we have $d(f^{-n}(x), f^{-n}(y)) \to 0$  (resp. $d(f^n(x), f^n(y)) \to 0$) as $n \to +\infty$. A collection of embedded $C^1$ discs $\Gamma = \{\gamma_i\}_{i \in \mathcal I}$ is a \emph{continuous family of unstable discs} if there is a Borel subset $K^s \subset M$ and a homeomorphism $\Phi : K^s \times D^u \to \union_i \gamma_i$, where $D^u \subset \R^d$ is the closed unit disc for some $d < \dim M$, satisfying: 
	\begin{itemize}
		\item The assignment $x \mapsto \Phi|_{\{x\} \times D^u}$ is a continuous map from $K^s$ to the space of $C^1$ embeddings $D^u \hookrightarrow M$, and this assignment can be extended to the closure $\cl(K^s)$; 
		\item For every $x \in K^s$, $\gamma = \Phi(\{x\} \times D^u)$ is an unstable disc in $\Gamma$.
	\end{itemize}
	Thus the index set $\mathcal I$ may be taken to be $K^s \times \{0\} \subset K^s \times D^u$. We define \emph{continuous families of stable discs} analogously. 
	
	A subset $\Lambda \subset M$ has \emph{hyperbolic product structure} if there is a continuous family $\Gamma^u = \{\gamma^u_i\}_{i \in \mathcal I}$ of unstable discs and a continuous family $\Gamma^s = \{\gamma^s_j\}_{j \in \mathcal J}$ of stable discs such that
	\begin{itemize}
		\item $\dim \gamma^u_i + \dim\gamma^s_j = \dim M$ for all $i,j$; 
		\item the unstable discs are transversal to the stable discs, with an angle uniformly bounded away from 0; 
		\item each unstable disc intersects each stable disc in exactly one point; 
		\item $\Lambda = \big( \union_i \gamma^u_i\big) \cap \big(\union_j \gamma^s_j \big)$. 
	\end{itemize}
	
	A subset $\Lambda_0 \subset \Lambda$ is an \emph{s-subset} if it has hyperbolic product structure defined by the same family $\Gamma^u$ of unstable discs as $\Lambda$,  and a continuous subfamily of stable discs $\Gamma_0^s \subset \Gamma^s$. A \emph{u-subset} is defined analogously. 
	
	\begin{definition}
		A $C^{1+\alpha}$ diffeomorphism $f : M \to M$, with $M$ a compact Riemannian manifold, is a \emph{Young's diffeomorphism} if the following conditions are satisfied: 
		\begin{enumerate}[label=(Y\arabic*)]
			\item There exists $\Lambda \subset M$ (called the \emph{base}) with hyperbolic product structure, a countable collection of continuous subfamilies $\Gamma_i^s \subset \Gamma^s$ of stable discs, and positive integers $\tau_i$, $i \in \N$, such that the $s$-subsets
			\[
			\Lambda_i^s := \union_{\gamma \in \Gamma^s_i} \big(\gamma \cap \Lambda \big) \subset \Lambda
			\]
			are pairwise disjoint and satisfy:
			\begin{enumerate}[label=(\alph*)]
				\item \emph{invariance}: for $x \in \Lambda_i^s$, 
				\[
				f^{\tau_i}(\gamma^s(x)) \subset \gamma^s(f^{\tau_i}(x)), \quad \textrm{and} \quad f^{\tau_i}(\gamma^u(x)) \supset \gamma^u(f^{\tau_i}(x)),
				\]
				where $\gamma^{u,s}(x)$ denotes the (un)stable disc containing $x$; and, 
				\item \emph{Markov property}: $\Lambda_i^u := f^{\tau_i}(\Lambda_i^s)$ is a $u$-subset of $\Lambda$ such that for $x \in \Lambda_i^s$, 
				\[
				f^{-\tau_i}(\gamma^s(f^{\tau_i}(x)) \cap \Lambda_i^u) = \gamma^s(x) \cap \Lambda, \quad \textrm{and} \quad f^{\tau_i} (\gamma^u(x) \cap \Lambda_i^s) = \gamma^u(f^{\tau_i}(x)) \cap \Lambda. 
				\]
			\end{enumerate}
			\item For $\gamma^u \in \Gamma^u$, we have
			\[
			\mu_{\gamma^u}(\gamma^u \cap \Lambda) > 0, \quad \textrm{and} \quad \mu_{\gamma^u}\Big( \mathrm{cl}\big( \left(\Lambda \setminus \textstyle\union_i \Lambda_i^s\right) \cap \gamma^u\big)\Big) = 0,
			\]
			where $\mu_{\gamma^u}$ is the induced Riemannian leaf volume on $\gamma^u$ and $\mathrm{cl}(A)$ denotes the closure of $A$ in $\T^2$ for $A \subseteq \T^2$. 
			\item There is $a \in (0,1)$ so that for any $i \in \N$, we have:
			\begin{enumerate}[label=(\alph*)]
				\item For $x \in \Lambda_i^s$ and $y \in \gamma^s(x)$, 
				\[
				d(F(x), F(y)) \leq ad(x,y);
				\]
				\item For $x \in \Lambda_i^s$ and $y \in \gamma^u(x) \cap \Lambda_i^s$, 
				\[
				d(x,y) \leq ad(F(x), F(y)),
				\]
			\end{enumerate}
			where $F : \union_i \Lambda_{i}^s \to \Lambda$ is the \emph{induced map} defined by 
			\[
			F|_{\Lambda^s_i} := f^{\tau_i}|_{\Lambda^s_i}.
			\]
			\item Denote $J^u F(x) = \det\big(DF|_{E^u(x)}\big)$. There exist $c > 0$ and $\kappa \in (0,1)$ such that: 
			\begin{enumerate}[label=(\alph*)]
				\item For all $n \geq 0$, $x \in F^{-n}\left(\union_i \Lambda_i^s\right)$ and $y \in \gamma^s(x)$, we have 
				\[
				\left| \log \frac{J^u F(F^n(x))}{J^u F(F^n(y))}\right| \leq c\kappa^n;
				\]
				\item For any $i_0, \ldots, i_n \in \N$ with $F^k(x), F^k(y) \in \Lambda^s_{i_k}$ for $0 \leq k \leq n$ and $y \in \gamma^u(x)$, we have 
				\[
				\left| \log\frac{J^u F(F^{n-k}(x))}{J^u F(F^{n-k}(y))}\right| \leq c\kappa^k.
				\]
			\end{enumerate}
			\item There is some $\gamma^u \in \Gamma^u$ such that 
			\[
			\sum_{i=1}^\infty \tau_i \mu_{\gamma^u} \left(\Lambda_i^s\right) < \infty. 
			\]
		\end{enumerate}
	\end{definition}
	
	
	We say the tower satisfies the \emph{arithmetic condition} if the greatest common divisor of the integers $\{\tau_i\}$ is 1. 
	
	We use the following result to discuss thermodynamics of Young's diffeomorphisms, which was originally presented as Proposition 4.1 and Remark 4 in \cite{PSZ17}. 
	
	\begin{proposition}\label{PSZ-4.1}
		Let $f : M \to M$ be a $C^{1+\alpha}$ diffeomorphism of a compact smooth Riemannian manifold $M$ satisfying conditions (Y1)-(Y5), and assume $\tau$ is the first return time to the base of the tower. Then the following hold: 
		\begin{enumerate}[label=(\arabic*)]
			\item There exists an equilibrium measure $\mu_1$ for the potential $\phi_1$, which is the unique SRB measure. 
			\item Assume that for some constants $C>0$ and $0 < h < h_{\mu_1}(f)$, with $h_{\mu_1}(f)$ the metric entropy, we have
			\[
			S_n := \# \left\{ \Lambda_i^s : \tau_i = n \right\} \leq Ce^{hn}
			\]
			Then there exists $t_0 < 0$ so that for every $t \in (t_0, 1)$, there exists a measure $\mu_t \in \mathcal M(f,Y)$, where $Y = \left\{f^k(x) : x \in \union \Lambda_i^s, \: 0 \leq k \leq \tau(x)-1 \right\}$, which is a unique equilibrium measure for the potential $\phi_t$. 
			\item Assume that the tower satisfies the arithmetic condition, and that there is $K > 0$ such that for every $i \geq 0$, every $x,y \in \Lambda_i^s$, and any $j \in \{0, \ldots, \tau_i\}$, 
			\begin{equation}\label{4.2 bound}
			d\left(f^j(x), f^j(y)\right) \leq K\max\{d(x,y), d(F(x), F(y))\}.
			\end{equation}
			Then for every $t_0 < t < 1$, the measure $\mu_t$ has exponential decay of correlations and satisfies the central limit theorem with respect to a class of functions which contains all H\"older continuous functions on $M$. 
		\end{enumerate}
	\end{proposition}
	
	\section{Tower representations of almost Anosov diffeomorphisms}
	
	Assume our almost Anosov map $f : \T^2 \to \T^2$ is nondegenerate and satisfies Assumption \ref{A}. Let $\tilde f$ be an Anosov diffeomorphism topologically conjugate to $f$ (the existence of such a map is given by Theorem \ref{Anosov-conjugacy}), and assume without loss of generality that $\tilde f$ is a linear automorphism. Let $\tilde\Ps$ be a finite Markov partition for $\tilde f$, and let $\tilde P \in \tilde\Ps$ be a partition element not intersecting $B_{r_0}(0)$. For $x \in \tilde P$, denote $\tilde\gamma^s(x)$ and $\tilde\gamma^u(x)$ respectively to be the connected component of the intersection of the stable and unstable leaves with $\tilde P$. We will construct a Young tower for the Anosov system $\tilde f: \T^2 \to \T^2$. The construction comes from Section 6.1 of \cite{PSZ17}, but we restate it here for the reader's convenience. 
	
	Let $\tilde\tau(x)$ be the first return time of $x$ to $\mathrm{Int}\tilde P$ for $x \in \tilde P$. For $x$ with $\tilde \tau(x) < \infty$, define:
	\[
	\tilde\Lambda^s(x) = \union_{y \in \tilde U^u(x) \setminus \tilde A^u(x)} \tilde\gamma^s(y),
	\]
	where $\tilde U^u(x) \subseteq \tilde\gamma^u(x)$ is an interval containing $x$, open in the induced topology of $\tilde\gamma^u(x)$, and $\tilde A^u(x) \subset \tilde U^u(x)$ is the set of points that either lie on the boundary of the Markov partition, or never return to $\tilde P$. Observe $\tilde{\Lambda}^s(x)$ is also expressible as 
	\[
	\tilde\Lambda^s(x) = \union_{y \in \tilde{\gamma}^s(x)} \theta_y\left(\tilde U^u(x)\right)\setminus \theta_y \left(\tilde A^u(x)\right),
	\]
	where $\theta_y : \tilde{\gamma}^u(x) \to \tilde{\gamma}^u(y)$ is the map $z \mapsto [y,z]$ for each $y \in \tilde{\gamma}^s(x)$, attained by sliding a point $z \in \tilde{\gamma}^u(x)$ along $\tilde\gamma^s(z)$ to the intersection of $\tilde\gamma^s(z)$ with $\tilde\gamma^u(y)$ (see remark \ref{localproduct}). In particular, $\tilde{\gamma}^u(y) \cap \tilde \Lambda^s(x) = \theta_{[y,x]}\left( \tilde U^u(x)\right) \setminus \theta_{[y,x]}\left(\tilde A^u(x) \right)$ for $y \in \tilde{\Lambda}^s(x)$. One can show the leaf volume of $\tilde A^u(x)$ is 0, so the above expression for $\tilde{\Lambda}^s(x)$ implies the leaf volume of $\tilde\gamma^u(y) \cap \tilde\Lambda^s(x)$ is positive. We further choose our interval $U^u(x)$ so that
	\begin{itemize}
		\item for $y \in \tilde \Lambda^s(x)$, we have $\tilde\tau(y) = \tilde\tau(x)$; and, 
		\item for $y \in \tilde P$ with $\tilde\tau(x) = \tilde\tau(y)$, we have $y \in \tilde\Lambda(z)$ for some $z \in \tilde P$. 
	\end{itemize}
	One can show the image under $\tilde f^{\tilde\tau(x)}$ of $\tilde\Lambda^s(x)$ is a $u$-subset containing $\tilde f^{\tilde \tau(x)}(x)$, and that for $x, y \in \tilde P$ with finite return time, either $\tilde\Lambda^s(x)$ and $\tilde\Lambda^s(y)$ are disjoint or coinciding. As discussed in \cite{PSZ17}, this gives us a countable collection of disjoint sets $\tilde\Lambda^s_i$ and numbers $\tilde\tau_i$ for which the Anosov system $(\T^2, \tilde f)$ is a Young's diffeomorphism, with $s$-sets $\tilde\Lambda_i^s$, inducing times $\tilde\tau_i$, and tower base
	\[
	\tilde\Lambda := \union_{i =1}^\infty \tilde\Lambda_i^s.
	\]
	The details in verifying conditions (Y1) - (Y5) for this Anosov system are left to the reader; in particular, conditions (Y1), (Y3), and (Y4) are immediate for linear hyperbolic toral automorphisms. See the discussion in \cite{PSZ17}, Section 6.2. 
	
	\begin{lemma}
		There exists $h < \htop(\tilde f)$ such that $S_n \leq e^{hn}$, where $S_n$ is the number of $s$-sets $\tilde\Lambda^s_i$ with inducing time $\tilde\tau_i = n$. 
	\end{lemma}
	\begin{proof} See \cite{PSZ17}, Lemma 6.1. \end{proof}
	
	Let $g : \T^2 \to \T^2$ be the conjugacy map so that $f \circ g = g \circ \tilde f$, and let $\Ps = g(\tilde \Ps)$, $P = g(\tilde P)$. Then $\Ps$ is a Markov partition for the almost Anosov system $(\T^2, f)$, and $P$ is a partition element. By continuity of $g$, we may assume the elements of $\Ps$ have arbitrarily small diameter. Further let $\Lambda = g(\tilde \Lambda)$. Then $\Lambda$ has direct hyperbolic product structure with full length stable and unstable curves $\gamma^s(x) = g(\tilde\gamma^s(x))$ and $\gamma^u(x) = g(\tilde\gamma^u(x))$. Then $\Lambda^s_i = g(\tilde\Lambda_i^s)$ are $s$-sets and $\Lambda^u_i = g(\tilde\Lambda^u_i) = f^{\tau_i}(\Lambda_i^s)$, where $\tau_i = \tilde\tau_i$ for each $i$, and $\tau(x) = \tau_i$ whenever $x \in \Lambda_i^s$. 
	
	Let $\tilde B = g^{-1} \left(B_{r_0}(0)\right)$. Assuming $r_0$ is sufficiently small, there is an integer $Q > 0$ and a partition element $\tilde P$ (possibly after refining $\tilde\Ps$) so that $\tilde f^n(x) \not\in \tilde B$ for $n \in \{1, \ldots, Q\}$ whenever $x \in \tilde P$ or $x \in \tilde B^c \cap \tilde f(\tilde B)$. These properties carry over to the Markov partition element $P$ for the almost Anosov system, for the same integer $Q$. 
	
	\begin{theorem}\label{Young-construct}
		The collection of $s$-subsets $\Lambda_i^s = g(\tilde\Lambda_i^s)$ satisfies conditions (Y1) - (Y5), making the almost Anosov diffeomorphism $f : \T^2 \to \T^2$ a Young's diffeomorphism. 
	\end{theorem}
	
	\begin{proof}
		Condition (Y1) follows from the corresponding properties of the Anosov diffeomorphism $\tilde f$ since $g$ is a topological conjugacy. The fact that $\mu_{\gamma^u}\left(\gamma^u \cap \Lambda\right) > 0$ follows from the corresponding property for the $\tilde\gamma^u$ leaves. Suppose $x \in \cl\big(\left(\Lambda \setminus \union_i \Lambda_i^s\right)\cap \gamma^u\big)$. Then either $x$ lies on the boundary of the Markov partition element $P$, or $\tau(x) = \infty$, and since both the Markov partition boundary and the set of $x \in P$ with $\tau(x) = \infty$ are Lebesgue null, we get condition (Y2). Condition (Y5) follows from Kac's formula, since the inducing times are first return times to the base of the tower. 
		
		To prove condition (Y3), let $x \in \Lambda_i^s$ and let $y \in \gamma^s(x)$. Since $P \cap B_{r_1}(0) = \emptyset$, $Df_x|_{E^s(x)}$ uniformly contracts vectors for every $x \in P$, which means $d(f(x), f(y)) \leq ad(x,y)$ for some $a \in (0,1)$. By increasing $r_1$, we may assume $f^n(y) \in B_{r_1}(0)$ if and only if $f^n(x) \in B_{r_1}(0)$ for $n \geq 1$. In Section 4 of \cite{Hu00}, the author proves the existence of local stable and unstable manifolds $W^s_\epsilon(x) = \{y \in \T^2 : d(f^n(x), f^n(y)) \leq \epsilon \: \forall n \geq 1\}$ at every $x \in \T^2$, and in particular for $x \in B_{r_1}(0)$. Letting $k = \min\{n \geq 1 : f^n(x) \in B_{r_1}(0)\}$ and $\epsilon = d(f^k(x), f^k(y))$, as long as the orbits of $x$ and $y$ are inside $B_{r_1}(0)$, the iterates do not exceed a distance of $\epsilon$ from each other. Therefore for all $n \geq k$ with $f^n(x), f^n(y) \in B_{r_1}(0)$, we have
		\[
		d(f^n(x), f^n(y)) \leq d(f^k(x), f^k(y)) \leq a^kd(x,y).
		\]
		Upon exiting $B_{r_1}(0)$, the trajectories of $x$ and $y$ continue to contract, and they do not expand inside of $B_{r_1}(0)$. Therefore $d(f^n(x), f^n(y)) \leq ad(x,y)$ for all $n \geq 1$, and in particular for $n = \tau(x) = \tau(y)$. This proves (Y3)(a), and (Y3)(b) is proven similarly by considering $f^{-1}$ instead of $f$. 
		
		Proving bounded distortion in condition (Y4) requires the following result: 
		
		\begin{lemma}\label{Hu-distortion}
			There exists a constant $I > 0$ and $\theta \in (0,1)$ such that if $\gamma \subset f(B_{r_1}(0)) \setminus B_{r_1}(0)$ is a $W^s$-segment (that is $\gamma$ is a subset of a stable leaf, and is homeomorphic to an open interval in the induced topology), and if $f^i(\gamma) \subset B_{r_1}(0)$ for $i = 1, \ldots, n-1$, then for every $ x, y \in \gamma$, 
			\begin{equation}\label{bounded-distortion}
			\left| \log \frac{\left|Df^{n}|_{E^u(x)}\right|}{\left|Df^n|_{E^u(y)}\right|}\right| \leq Id^u(x,y)^\theta,
			\end{equation}
			where $d^u(x,y)$ is the induced Riemannian distance from $x$ to $y$ in the stable leaf $\gamma$. 
		\end{lemma}
		\begin{proof}
			See Lemma 7.4 in \cite{Hu00}. 
		\end{proof}
		\begin{remark}\label{Hu-distortion-remark}
			This lemma is proven originally for $\gamma$ an interval in an unstable leaf and for $Df^{-n}$ instead of $Df^n$, but the argument is effectively the same by going forwards instead of backwards. Additionally, by H\"older continuity of the stable foliation outside of $B_{r_1}$, by changing $I$ and $\theta$, we can replace $d^u(x,y)$ with $d(x,y)$. 
		\end{remark}
		
		Define the \emph{itinerary} $\mathcal I(x) = \{0 = n_0 < n_1 < \cdots < n_{2\ell+1} = \tau(x)\} \subset \Z$ of a point $x \in \Lambda$, with $\ell = \ell(x)$, so that $f^k(x) \in B_{r_1}(0)$ if and only if $n_{2i-1} < k < n_{2i}$ for $i \geq 1$. Assume $\Lambda$ is small enough so that $\mathcal I(x) = \mathcal I(y)$ whenever $y \in \gamma(x) \subset \Lambda$. When we have $n_{2i} \leq k \leq n_{2i+1}$, $i \geq 0$, $f^k(x)$ is outside of $B_{r_1}(0)$. Since $f$ is a linear Anosov map outside of $B_{r_1}(0)$, for a fixed $n \geq 1$, the determinant of $Df^n_x |_{E^u(x)} : E^u(x) \to E^u(f^n(x))$ is constant for all $x$ with $f^k(x) \not\in B_{r_1}(0)$ for $0 \leq k \leq n$. Therefore, for each $i = 0, \ldots, 2\ell + 1$, 
		\[
		\left| \log\frac{\left| Df^{n_{2i+1} - n_{2i}}|_{E^u(f^{n_{2i}}(x))}\right|}{\left| Df^{n_{2i+1} - n_{2i}}|_{E^u(f^{n_{2i}}(y))}\right|}\right| = 0.
		\]
		Meanwhile, when $n_{2i-1} < k < n_{2i}$, since $f^{2i-1}(x)$ and $f^{2i-1}(y)$ lie in the same stable leaf, and $f^k\left(f^{n_{2i-1}}(x)\right), f^k\left(f^{n_{2i-1}}(y)\right) \in B_{r_1}(0)$ for $k = 1, \ldots, n_{2i} - n_{2i-1} - 1$, then lemma \ref{Hu-distortion} and remark \ref{Hu-distortion-remark} gives us: 
		\[
		\left| \log\frac{\left| Df^{n_{2i} - n_{2i-1}}|_{E^u(f^{n_{2i-1}}(x))}\right|}{\left| Df^{n_{2i} - n_{2i-1}}|_{E^u(f^{n_{2i-1}}(y))}\right|}\right|  \leq Id\left(f^{n_{2i-1}}(x), f^{n_{2i-1}}(y)\right)^\theta \leq Ia^{\theta n_{2i-1}}d(x,y)^\theta. 
		\]
		Therefore since $d(x,y) < 1$ for $x, y \in \Lambda$ with $y \in \gamma^s(x)$, and $0 < a < 1$, we have: 
		\begin{align*}
		\left|\log\frac{J^u F(x)}{J^u F(y)} \right| &= \left| \log \prod_{i=1}^{2\ell+1} \frac{\left|Df^{n_i - n_{i-1}}|_{E^u\left(f^{n_{i-1}}(x)\right)}\right|}{\left|Df^{n_i - n_{i-1}}|_{E^u\left(f^{n_{i-1}}(y)\right)}\right|}\right| \\
		&\leq \sum_{i=1}^{2\ell+1} \left| \log \frac{\left|Df^{n_i - n_{i-1}}|_{E^u\left(f^{n_{i-1}}(x)\right)}\right|}{\left|Df^{n_i - n_{i-1}}|_{E^u\left(f^{n_{i-1}}(y)\right)}\right|} \right| \\
		&\leq I \sum_{i=1}^{\ell} a^{\theta n_{2i-1}} d(x,y)^\theta  \\
		&\leq Id(x,y)^\theta \sum_{n=0}^\infty a^{\theta n} \\
		&= \frac{I}{1-a^\theta} d(x,y)^\theta. 
		\end{align*}
		By condition (Y3), for $n \geq 0$, $x \in F^{-n}\left( \union_i \Lambda_i^s\right)$ and $y \in \gamma^s(x)$, we have: 
		\begin{align*}
		\left| \log\frac{ J^u F(F^n(x))}{J^u F(F^n(y))} \right| \leq \frac I{1-a^\theta} d\left(F^n(x), F^n(y)\right)^\theta \leq \left( \frac{I}{1-a^\theta}\right) a^{\theta n}.
		\end{align*}
		Condition (Y4)(a) now holds with $c = I/(1-a^\theta)$ and $\kappa = a^\theta$. Condition (Y4)(b) can be proven in a similar manner by working backwards instead of forwards. 
	\end{proof}
	
	\section{Proof of Theorem 3.1}
	
	The conjugacy map $g$ preserves topological and combinatorial properties of the Anosov map $\tilde f$, so the number $S_n$ of $s$-sets $\tilde\Lambda_i^s$ with inducing time $\tilde\tau_i = \tau_i$ is the same as the number of $s$-sets $\Lambda_i^s$ for the almost Anosov map $f$ with inducing time $\tau_i$. Therefore, by Lemma 5.1, $S_n \leq e^{hn}$ for some $h < \htop(\tilde f) = \htop(f)$. 
	
	By proposition \ref{PSZ-4.1}(1) and theorem \ref{Young-construct}, there is a unique SRB measure $\mu_1$, which is an equilibrium measure for $\phi_1$. We claim that $h < h_{\mu_1}(f)$. Indeed, for the linear Anosov map $\tilde f$, we have $h_m(\tilde f) = \htop(\tilde f)$, where $m$ is the Lebesgue measure. For $r_1$ sufficiently small, we have 
	\[
	\left| \int_{\T^2} \log \left| Df|_{E^u}\right| \, d\mu_1 - \log\lambda\right| < \epsilon
	\]
	for small $\epsilon > 0$, where $\log\lambda = \sup \log\big|D\tilde f|_{E^u}\big|$ (since $\log \left| Df|_{E^u}\right| = \log\lambda$ outside of $B_{r_1}(0)$). By the Pesin entropy formula, it follows that $\left|h_{\mu_1}(f) - h_{m}(\tilde f)\right|< \epsilon$, and our claim holds. Therefore, by Proposition \ref{PSZ-4.1}(2), there is a $t_0 = t_0(P)$ so that for each $t_0 < t < 1$, there is a unique equilibrium measure $\mu_t \in \mathcal M(f, Y)$, with $Y$ the Young tower. 
	
	Since the Anosov map $\tilde f : \T^2 \to \T^2$ is Bernoulli, every power of $\tilde f$ is ergodic, so the tower for $\tilde f$ satisfies the arithmetic condition. This combinatorial property is preserved under topological conjugacy, so the almost Anosov map $f$ also satisfies the arithmetic condition. 
	
	Note if $x, y \in \Lambda_i^s$, with $y \in \gamma^s(x)$, the distance $d(f^j(x), f^j(y))$ is decreasing as $j \to \infty$. On the other hand, if $y \not\in \gamma^s(x)$, then $d(f^j(x), f^j(y))$ is increasing, but its maximal value is $\leq \diam(P)$ and is attained at $j = \tau_i$. Therefore condition (\ref{4.2 bound}) is satisfied. By Proposition \ref{PSZ-4.1}(3), $\mu_t$ has exponential decay of correlations and satisfies the CLT with respect to a class of functions which contains all H\"older continuous functions on $M$. 
	
	Suppose $\hat P$ is another element of the Markov partition of $(\T^2, f)$ for which $f^n(x) \not\in B_{r_0}(0)$ for $1 \leq n \leq Q$ whenever $x \in \hat P$ or $x \in f(B_{r_0}(0)) \setminus B_{r_0}(0)$. Arguing as above, we find $\hat t_0 < 0$ such that for $\hat t_0 < t < 1$, there exists a unique equilibrium measure $\hat \mu_t$ associated to the geometric $t$-potential among all measures $\mu$ for which $\mu(\hat P) > 0$. Note $\mu_t(U) > 0$ and $\hat \mu_t(\hat U) > 0$ for every pair of open sets $U \subset P$ and $\hat U \subset \hat P$. Since $f$ is topologically transitive, for every such pair of open sets, there is an integer $k > 0$ such that $f^k(U) \cap \hat U \neq \emptyset$. By uniqueness, it follows that $\mu_t = \hat\mu_t$. 
	
	To prove Statement 2 of Theorem \ref{main-theorem}, suppose $\mu$ is an invariant Borel probability measure. Assume $\mu$ is ergodic. By the Margulis-Ruelle inequality, 
	\[
	h_\mu(f) \leq \int_{\T^2} \log\left|Df|_{E^u(x)}\right| \, d\mu(x) = -\int_{\T^2} \phi_1 \, d\mu.
	\]
	Hence $h_\mu(f) + \int \phi_1 \, d\mu \leq 0$. If $\mu$ has only zero Lyapunov exponents (in particular, at almost every $x \in \T^2$ there are no positive Lyapunov exponents), then $\log\left|Df|_{E^u(x)}\right| = 0$ $\mu$-a.e. The only point at which $\log\left|Df|_{E^u(x)}\right| = 0$ is $x=0$, so $\mu = \delta_0$. In this instance, we have $h_{\delta_0}(f) + \int \phi_1 \, d\delta_0 = 0$, so $P(\phi) = 0$ and $\delta_0$ is an equilibrium state for $\phi_1$. 
	
	On the other hand, \cite{Hu00} guarantees the existence of an invariant finite SRB measure $\mu$ for $f$. In particular, $\mu$ is a smooth measure, so by the Pesin entropy formula, $h_\mu(f) + \int \phi_1 \,d\mu = 0$, so $\mu$ is also an equilibrium measure. Any other equilibrium measure with positive Lyapunov exponents also satisfies the entropy formula. By \cite{LY84}, such a measure is also an SRB measure, and by \cite{RH2TU}, this SRB measure is unique.
	
	Finally, to prove Statement 3 of Theorem \ref{main-theorem}, fix $t > 1$, and let $\mu$ be an ergodic measure for $f$. Again, by the Margulis-Ruelle inequality, $$h_\mu(f) \leq t\int \log\left|Df|_{E^u(x)} \right|\,d\mu,$$ with equality if and only if $\int \log\left|Df|_{E^u(x)} \right|\,d\mu = 0$. In particular, we have equality if and only if $\mu$ has zero Lyapunov exponents $\mu$-a.e. As we saw, the only measure satisfying this is $\mu=\delta_0$, so $h_\mu(f) + \int \phi_t \,d\mu \leq 0$, with equality only for $\mu=\delta_0$. 
	
	\section*{Acknowledgments} I would like to thank Penn State University and the Anatole Katok Center for Dynamical Systems and Geometry where this work was done. I also thank my advisor, Y. Pesin, for introducing me to this problem and for valuable input over the course of my investigation into almost Anosov systems. Additionally I would like to thank H. Hu and F. Shahidi for their helpful remarks. 
	

	\medskip
	Received January 2019; revised August 2019.
	\medskip
	
\end{document}